\documentclass[9pt, a4paper, reqno]{amsart}
 \usepackage[a4paper,margin=1.20in]{geometry}

\usepackage{graphicx}
\usepackage[pdftex, colorlinks=true, urlcolor=blue, linkcolor=blue, citecolor=blue, pdfstartview=FitH]{hyperref}

\usepackage{epsfig,color}
 \usepackage{hyperref}
 \usepackage{float}
 
\usepackage{mathrsfs}

  \usepackage{setspace}
 
\newtheorem{thm}{Theorem}

\newtheorem{cor}[thm]{Corollary}
\newtheorem{prop}[thm]{Proposition}

\theoremstyle{definition}
\newtheorem{rmk}{Remark}[section]
\newtheorem{exmp}{Example}[section]

\numberwithin{equation}{section}

\title[Wirtinger-type inequalities and geometric applications]{Higher order Wirtinger-type inequalities \\ and sharp bounds for the isoperimetric deficit}

\author{Kwok-Kun Kwong}
\address{School of Mathematics and Applied Statistics, University of Wollongong, NSW 2522, Australia}
\email{kwongk@uow.edu.au}

\author{Hojoo Lee}
\address{Department of Mathematics and Institute of Pure and Applied Mathematics, Jeonbuk National University,
Jeonju 54896, Korea}
\email{compactkoala@gmail.com, kiarostami@jbnu.ac.kr}

 \numberwithin{equation}{section}
 
\begin{document}

 \maketitle

\begin{abstract}
Using Fourier analysis, we derive Wirtinger-type inequalities of arbitrary high order. As applications, we prove various sharp geometric inequalities for closed curves on the Euclidean plane. In particular, we obtain both sharp lower and upper bounds for the isoperimetric deficit.
\end{abstract}


 

\section{Introduction}

The isoperimetric inequality states that for a simple closed curve $\mathcal C$ of length $L$ on the plane which encloses a region of area $A$, the isoperimetric deficit $L^2-4\pi A$ is
 non-negative, and it is zero if and only if $\mathcal C$ is a circle. In 1902, A. Hurwitz gave two Fourier analytic proofs of the isoperimetric inequalities, one for convex curves 
 (\cite[pp. 371-373]{Hurwitz1902}, \cite[Section 4.2]{Groemer1996}) and another one for general curves (\cite[pp. 392--394]{Hurwitz1902}, \cite[Section 4.1]{Groemer1996}). As well-known in \cite[Chapter 7, Section 7]{HLP1988}, \cite[Chapter 1, Section 2]{Chavel2001}, and \cite[pp.1183--1185]{Osserman1978}, the planar isoperimetric inequality $L^2-4\pi A \geq 0$ is an immediate consequence of the one dimensional Poincar\'{e} inequality, so called Wirtinger's inequality \cite[Section 23, II. Ein Lemma von Wirtinger]{Blaschke1949},
\begin{equation} \label{introWir01}
\int_{0}^{\, 2\pi \,} {f'(t)}^{2} dt - \int_{0}^{\, 2\pi \,} {f(t)}^{2} dt  \geq 0,
\end{equation}
which holds when the smooth function $f: \mathbb{R} \to \mathbb{R}$ is $2\pi$-periodic and has zero mean $\int_{0}^{\, 2\pi \,} f(t) dt =0$.

Wirtinger's inequality is ubiquitous in geometric and functional analysis. For its several applications in geometric flows, see \cite{GageHamilton1986}, \cite[Section 4]{LinTsai2014},
and \cite[Lemma 2.8 and Remark 8]{Tsai2005}. For various proofs of the Wirtinger inequality, we refer the interested readers to \cite[Chapter 5, Section 10]{BeckenbachBellman1961}, 
\cite[Chapter 4, Section 2]{DAngelo2019}, and \cite[Chapter 2]{MPF1991}. Inspired by the ingenious method of Hurwitz using Fourier series, we will derive a family of Wirtinger-type inequalities, which involves the derivatives of $f$ of arbitrarily large order. If a smooth function $f: \mathbb{R} \to \mathbb{R}$ is $2\pi$-periodic and has zero mean, then 
\begin{equation} \label{introWir02}
\int_{0}^{\, 2\pi \,} {f''(t)}^{2} dt - 5 \int_{0}^{\, 2\pi \,} {f'(t)}^{2} dt + 4 \int_{0}^{\, 2\pi \,} {f(t)}^{2} dt \geq 0,
\end{equation}
and
\begin{equation} \label{introWir03}
\int_{0}^{\, 2\pi \,} {f'''(t)}^{2} dt - 14 \int_{0}^{\, 2\pi \,} {f''(t)}^{2} dt + 49  \int_{0}^{\, 2\pi \,} {f'(t)}^{2} dt  -36 \int_{0}^{\, 2\pi \,} {f(t)}^{2} dt \geq 0.
\end{equation}

It has been of great interest to obtain upper and lower bounds of the isoperimetric deficit in terms of other geometric quantities  which are computable. The first geometric application of our second order inequality (\ref{introWir02}) is the reverse Sachs inequality, which is the second inequality in the following:
\begin{equation} \label{introSachs02}
0 \leq \frac{\, {L}^{3} \, }{\, 4 {\pi}^{2} \, } - \int_{\mathcal{C}} {\left\vert \mathbf{X} -  \mathbf{G}  \, \right\vert}^{2} \, ds \leq \frac{ L^4 }{\,64 {\pi}^{4}\,}
\left(\int_{\mathcal{C}} {\, \kappa \,}^{2} \, ds - \frac{4 \pi^2}{L} \right),
\end{equation}
where $\mathbf{G}$ is the center of the gravity of the curve $\mathcal{C}$ and $\kappa$ denotes the curvature function on the curve $\mathcal{C}$.
The inequalities in (\ref{introSachs02}) give the lower and upper bounds of the isoperimetric deficit $L^2 - 4 \pi A$:
\begin{equation} \label{Reverse01intro1}
 \frac{\, 2 {\pi}^{2} \, }{L} \int_{\mathcal{C}} \, {\left\vert \, \mathbf{X} - \mathbf{G} - {\left(\frac{L}{\, 2\pi \,} \right)} \mathbf{N} \, \right\vert}^{2} \, ds
 \leq L^2 - 4 \pi A
\end{equation}
and
\begin{equation} \label{Reverse01intro2}
L^2 - 4 \pi A
\leq   \frac{\, 2 {\pi}^{2} \, }{L} \int_{\mathcal{C}} \, {\left\vert \, \mathbf{X} - \mathbf{G} - {\left(\frac{L}{\, 2\pi \,} \right)} \mathbf{N} \, \right\vert}^{2} \, ds + \frac{\, 2 {\, \pi \,}^{2} \, }{3L} \int_{\mathcal{C}} \, {\left\vert \, \mathbf{X} - \mathbf{G} - {\left(\frac{L}{\, 2\pi \,} \right)}^{2} \kappa \mathbf{N} \, \right\vert}^{2} \, ds.
\end{equation}

The first inequality in (\ref{introSachs02}), called the Sachs inequality \cite{Sachs1960}, is an immediate consequence of the Wirtinger inequality (\ref{introWir01}).
As illustrated in Osserman's outstanding survey \cite[pp.1203--1204]{Osserman1978}, the Sachs inequality implies the isoperimetric inequality.
This technique could be extended to the isoperimetric inequalities for curves on two dimensional minimal surfaces \cite{Chakerian1978, DHT2010, LSY1984, Osserman1978}.
In particular, G. D. Chakerian \cite{Chakerian1978} proved that the sharpened isoperimetric inequality \eqref{Reverse01intro1} holds on minimal surfaces.
For higher dimensional generalizations, we refer the interested readers to Chavel \cite{Chavel1978}, Reilly \cite{Reilly1977}, and
\cite[Theorem 5.8]{Osserman1978}.

As another application of the higher order Wirtinger's inequality, assuming the convexity of the curve $\mathcal{C}$, we prove, among other things, the geometric inequalities (Example \ref{coro second main})
\begin{equation} \label{introLT}
0 \leq \left(\int_{\mathcal{C}} \frac{1}{\, \kappa\, } ds - \frac{L^2}{\, 2\pi \,} \right) - \frac{3}{\, 2\pi\, } \left(L^{2} - 4\pi A \right)
\leq \frac{1}{\, 12\, } \left[ \, \int_{\mathcal{C}} \frac{1}{\, {\, \kappa \,}^{5} \, } {\left(\frac{d\kappa}{ds} \right)}^{2} ds - \frac{6}{\, \pi\, } \left(L^{2} - 4\pi A \right) \, \right].
\end{equation}
The first inequality in (\ref{introLT}) is also a consequence of our second order Wirtinger-type inequality (\ref{introWir02}), and is
equivalent to the Lin-Tsai inequality \cite[Lemma 1.7]{LinTsai2012}. The second inequality in (\ref{introLT}) is a consequence of our
third order Wirtinger-type inequality (\ref{introWir03}), and can be regarded as a lower bound of the isoperimetric deficit. Both inequalities in \eqref{introLT} possess non-round equality cases and we will discuss it in Section \ref{Gapplications}. The two inequalities in (\ref{introLT}) imply that the convex curve $\mathcal{C}$ satisfies the reversed isoperimetric inequality
\[
L^{2} - 4\pi A \leq \frac{\, \pi\, }{6} \int_{\mathcal{C}} \frac{1}{\, {\, \kappa \,}^{5} \, } {\left(\frac{d\kappa}{ds} \right)}^{2} ds.
\]
We also establish that the convex curve $\mathcal{C}$ satisfies the reverse isoperimetric inequalities (Corollary \ref{cor reverse})
\[
L^{2} - 4\pi A \leq \frac{L^2}{\, 24\pi \,} \int_{\mathcal{C}} \frac{1}{\, {\, \kappa \,}^{3} \, } {\left(\frac{d\kappa}{ds} \right)}^{2} ds \quad
\text{and} \quad
L^{2} - 4\pi A \leq \frac{AL}{\, 4\pi \,} \int_{\mathcal{C}} \frac{1}{\, {\, \kappa \,}^{2} \, } {\left(\frac{d\kappa}{ds} \right)}^{2} ds.
\]

\section{Higher order Wirtinger-type inequalities}

Throughout this paper, ${f}^{(k)}(t)$ denotes the $k$-th derivative of the function $f(t)$ with respect to t.
In this section, we will prove various higher order Wirtinger-type inequalities.

\begin{prop} [\textbf{Higher order Wirtinger inequalities}] \label{Wirs0}
Let $m \in \mathbb N $ and $f : \mathbb{R} \to \mathbb{R}$ be a $2\pi$-periodic function of class $C^{m}$ with zero mean
\begin{equation} \label{zero mean0}
\int_{0}^{\, 2\pi \,} \, f(t) \, dt=0.
\end{equation}
\begin{enumerate}
\item[\textbf{(a)}]
We have the inequality
\begin{equation} \label{wirhigherfrac}
0 \leq \sum_{k=0}^{m} {\mathbf{c}}_{m, k} \int_{0}^{\, 2\pi \,} { {f}^{(k)} (t) }^{2} \, dt.
\end{equation}
Here, ${\mathbf{c}}_{m, k}$ are the central factorial numbers of even indices of the first kind \cite[(1.11)]{GelineauZeng2010} , explicitly given by the coefficients of the $m$-th degree polynomial
\begin{equation*} \label{prodpolynomial1}
Q_{m}(t) := \prod_{j=1}^{m} \left(t-{j}^{2} \right)
= \sum_{k=0}^{m} {\mathbf{c}}_{m, k} {t}^{k}.
\end{equation*}
\item[\textbf{(b)}] We have the inequality
\begin{equation} \label{wirhigher0}
0 \leq \sum_{k=0}^{m-1} {\lambda}_{m, k} \int_{0}^{\, 2\pi \,} \left[ { {f}^{(k+1)} (t) }^{2} - { {f}^{(k)} (t) }^{2} \right] dt.
\end{equation}
Here, ${\lambda}_{m, k}$ are the coefficients of the $(m-1)$-th degree polynomial $P_{m}(t) $, defined by
\begin{equation*} \label{prodpolynomial2}
P_{1}(t)=1={\lambda}_{0, 0},  \quad P_{m}(t) := \prod_{j=2}^{m} \left(t-{j}^{2} \right)=\left(t-{2}^{2} \right) \cdots \left(t-{m}^{2} \right)= \sum_{k=0}^{m-1} {\lambda}_{m, k} {t}^{k},
\quad m \geq 2.
\end{equation*}
\item[\textbf{(c)}]
For $m \geq 2$, we have the inequality
\begin{equation} \label{wirhigher2}
0 \leq \prod_{j=2}^{m} \left(1 - j^2 \right) \int_{0}^{\, 2\pi \,} \left[ { {\dot{f}} (t) }^{2} - { {f}(t) }^{2} \right] dt + \sum_{k=1}^{m-1} {S}_{m, k} \int_{0}^{\, 2\pi \,} {\left({f}^{(k+1)} (t) + {f}^{(k-1)} (t) \right)}^{2} \, dt.
\end{equation}
Here, the constants ${S}_{m, 1}$, $\cdots$, ${S}_{m, m-1}$ are the coefficients of the $(m-2)$-th degree polynomial ${\mathcal{S}}_{m}(t) \in \mathbb{Z}\left[t\right]$, defined by
\begin{equation} \label{secondpolynomial}
{ \mathcal{S}}_{m}(t) := \frac{ P_{m}(t) - P_{m}(1) }{ t -1} = \sum_{k=1}^{m-1} \, {S}_{m, k} \, {t}^{k-1}.
\end{equation}
\end{enumerate}
The inequalities in (a), (b), (c) are equivalent. The equality holds in the three equivalent inequalities if and only if the function $f$ is the finite Fourier series of the form
\[
f(t)= \sum_{n=1}^{m} \left({\alpha}_{n} \cos \left(n t \right) + {\beta}_{n} \sin \left(n t \right)\right)
\]
for some constants ${\alpha}_{1}, \cdots, {\alpha}_{n}, {\beta}_{1}, \cdots, {\beta}_{n} \in \mathbb{R}$.
\end{prop}

\begin{proof}
The proof is inspired by the Hurwitz method using Fourier series. In the case where $m=1$, the inequalities in \eqref{wirhigherfrac} and \eqref{wirhigher0} are the classical
Wirtinger inequality. (For instance, see Chavel's monograph \cite[p. 8]{Chavel2001}.)
By the hypothesis \eqref{zero mean0}, the Fourier series of $f $ is given by
\begin{eqnarray*} \label{zero mean fourier 1}
f(t) \sim
\sum_{n\in \mathbb Z\setminus\{0\}} a_{n} e^{i n t}.
\end{eqnarray*}
More generally \cite[p. 43]{SteinShakarchi2011}, for $0\le k\le m$,
\begin{eqnarray*}
f^{(k)}(t) &\sim&
\sum_{n\in \mathbb Z\setminus\{0\}} a_{n} (in)^k e^{i n t}.
\end{eqnarray*}
By the Parseval identity \cite[p. 80]{SteinShakarchi2011}, we have
\begin{equation*} \label{zero mean fourier gernal}
\frac{1}{\,2\pi\,}  \int_{0}^{\, 2\pi \,} \, {{f}^{(k)}(t)}^{2} \, dt = \sum_{n\in\mathbb Z\setminus\{0\}} n^{2k} |a_n|^2.
\end{equation*}
Using the definition of $Q_{m}(t)$, we obtain the inequality \eqref{wirhigherfrac}
\begin{eqnarray*} \label{fracwirhigherproof}
\frac{1}{\,2\pi\,}  \sum_{k=0}^{m} {\mathbf{c}}_{m, k} \int_{0}^{\, 2\pi \,} { {f}^{(k)} (t) }^{2} \, dt
&=& \sum_{k=0}^{m} {\mathbf{c}}_{m, k} \sum_{n\in\mathbb Z\setminus\{0\}} n^{2k} |a_n|^2 \\
&=& \sum_{n\in\mathbb Z\setminus\{0\}} \, \left[ \sum_{k=0}^{m} {\mathbf{c}}_{m, k} n^{2k} \right] \, |a_n|^2 \\
&=& \sum_{n\in\mathbb Z\setminus\{0\}} Q_{m} \left(n^{2} \right) |a_n|^2 \\
&=& \sum_{n\in\mathbb Z\setminus\{0\}} \left(n^{2} - {1}^{2} \right) \left({n}^{2}-{2}^{2} \right) \cdots \left({n}^2-{m}^{2} \right) |a_n|^2 \\
&=&  \sum_{|n|\ge m+1} \left(n^{2} - {1}^{2} \right) \left({n}^{2}-{2}^{2} \right) \cdots \left({n}^2-{m}^{2} \right) |a_n|^2 \\
&\geq& 0.
\end{eqnarray*}
Similarly, using the definition of $P_{m}(t)$, we obtain the inequality \eqref{wirhigher0}
\begin{eqnarray*} \label{wirhigherproof}
\frac{1}{\,2\pi\,}  \sum_{k=0}^{m-1} \, \int_{0}^{\, 2\pi \,} \, {\lambda}_{m, k} \left[ \, { {f}^{(k+1)} (t) }^{2} - { {f}^{(k)} (t) }^{2} \, \right] \, dt
&=& \sum_{k=0}^{m-1} \, {\lambda}_{m, k} \, \sum_{n\in\mathbb Z\setminus\{0\}} \left(n^{2(k+1)} - n^{2k} \right) |a_n|^2 \\
&=& \sum_{n\in\mathbb Z\setminus\{0\}} \, \left[ \left(n^{2} - 1 \right) \sum_{k=0}^{m-1} \, {\lambda}_{m, k} \, n^{2k} \right] \, |a_n|^2 \\
&=& \sum_{n\in\mathbb Z\setminus\{0\}} \left(n^{2} - 1 \right) P_{m} \left(n^{2} \right) |a_n|^2 \\
&=& \sum_{n\in\mathbb Z\setminus\{0\}} \left(n^{2} - {1}^{2} \right) \left({n}^{2}-{2}^{2} \right) \cdots \left({n}^2-{m}^{2} \right) |a_n|^2 \\
&=&  \sum_{|n|\ge m+1} \left(n^{2} - {1}^{2} \right) \left({n}^{2}-{2}^{2} \right) \cdots \left({n}^2-{m}^{2} \right) |a_n|^2 \\
&\geq &0.
\end{eqnarray*}
Now, we verify the
inequality \eqref{wirhigher2}, which can be written as
\begin{equation} \label{wirhigher2here}
0 \leq \sum_{k=1}^{m-1} {S}_{m, k} \int_{0}^{\, 2\pi \,} {\left({f}^{(k+1)} (t) + {f}^{(k-1)} (t) \right)}^{2} \, dt
+ {S}_{m, 0} \int_{0}^{\, 2\pi \,} \left[ { {\dot{f}} (t) }^{2} - { {f}(t) }^{2} \right] dt.
\end{equation}
Here, ${S}_{m, 0} := \prod_{j=2}^{m} \left(1 - j^2 \right) = P_{m}(1)$. We will show that the inequality \eqref{wirhigher2here} is equivalent to
\eqref{wirhigher0}, i.e.
\[
\sum_{k=0}^{m-1} {\lambda}_{m, k} J_{k} \geq 0,
\]
where we define the integral
\[
J_{k} = \int_{0}^{\, 2\pi \,} \, \left[ \, { {f}^{(k+1)} (t) }^{2} - { {f}^{(k)} (t) }^{2} \, \right] \, dt.
\]
By the definitions of ${\mathcal{S}}_{m}(t)$ and $P_{m}(t)$, we obtain the telescoping series
\begin{equation*} \label{tele}
\sum_{k=0}^{m-1} {\lambda}_{m, k} {t}^{k}
= P_{m}(t) = \left(t -1 \right) {\mathcal{S}}_{m}(t) + P_{m}(1)
= \left(t -1 \right) {\mathcal{S}}_{m}(t) + {S}_{m, 0}
= \sum_{k=0}^{m-1} \, \left({S}_{m, k} - {S}_{m, k+1} \right) t^{k},
\end{equation*}
where ${S}_{m, m}:=0$. Comparing the coefficients in both sides yields
\begin{equation*} \label{rec03}
{\lambda}_{m, k} = {S}_{m, k} - {S}_{m, k+1}
\end{equation*}
for $k \in \{0, 1, \cdots, m-1\}$. Integration by parts gives
\[
\int_{0}^{\, 2\pi \,} {\left({f}^{(k+1)} (t) + {f}^{(k-1)} (t) \right)}^{2} \, dt
= \int_{0}^{\, 2\pi \,} \, \left[ \, { {f}^{(k+1)} (t) }^{2} - { {f}^{(k)} (t) }^{2} \, \right] \, dt -
\int_{0}^{\, 2\pi \,} \, \left[ \, { {f}^{(k)} (t) }^{2} - { {f}^{(k-1)} (t) }^{2} \, \right] \, dt = J_{k} - J_{k-1}.
\]
Using the fact that $S_{m,m}=0$, we have the telescoping series 
\begin{eqnarray*}
0&\le&\sum_{k=0}^{m-1} {\lambda}_{m, k} J_{k} =  \sum_{k=0}^{m-1} \left({S}_{m, k} - {S}_{m, k+1} \right) J_{k}
=  S_{m,0}J_0+\sum_{k=1}^{m-1}S_{m,k}(J_k-J_{k-1}) \\
&=& \sum_{k=1}^{m-1} {S}_{m, k} \int_{0}^{\, 2\pi \,} {\left({f}^{(k+1)} (t) + {f}^{(k-1)} (t) \right)}^{2} \, dt
+ {S}_{m, 0} \int_{0}^{\, 2\pi \,} \left[ { {\dot{f}} (t) }^{2} - { {f}(t) }^{2} \right] dt.
\end{eqnarray*}
This completes the proof of \eqref{wirhigher2here}. The above also shows the equalities
\begin{equation*}
\begin{split}
\sum_{k=0}^{m} {\mathbf{c}}_{m, k} \int_{0}^{\, 2\pi \,} { {f}^{(k)} (t) }^{2} \, dt =& \sum_{k=0}^{m-1} \, \int_{0}^{\, 2\pi \,} \, {\lambda}_{m, k} \left[ \, { {f}^{(k+1)} (t) }^{2} - { {f}^{(k)} (t) }^{2} \, \right] \, dt\\
=&\sum_{k=1}^{m-1} {S}_{m, k} \int_{0}^{\, 2\pi \,} {\left({f}^{(k+1)} (t) + {f}^{(k-1)} (t) \right)}^{2} \, dt
+ {S}_{m, 0} \int_{0}^{\, 2\pi \,} \left[ { {\dot{f}} (t) }^{2} - { {f}(t) }^{2} \right] dt.
\end{split}
\end{equation*}
We conclude that the three inequalities are equivalent.
\end{proof}

\begin{prop} \label{Wirs}
Let $f : \mathbb{R} \to \mathbb{R}$ be a $2\pi$-periodic $C^{m}$ function with zero mean.
 With the assumptions and notation of Proposition \ref{Wirs0}, the following inequalities hold.
\begin{enumerate}
\item[\textbf{(a)}]
\begin{equation}\label{wirhigher a}
0\le
\sum_{k=0}^{m}\mathbf{c}_{m,k}\int_{0}^{\, 2\pi \,}{f^{(k)}(t)}^2dt
\le
\frac{1}{(m+1)^2}\sum_{k=0}^{m}\mathbf{c}_{m,k}\int_{0}^{\, 2\pi \,}{f^{(k+1)}(t)}^2 dt.
\end{equation}
\item[\textbf{(b)}]
\begin{equation*} \label{wirhigher}
0 \leq \sum_{k=0}^{m-1} {\lambda}_{m, k} \int_{0}^{\, 2\pi \,} \left[ { {f}^{(k+1)} (t) }^{2} - { {f}^{(k)} (t) }^{2} \right] dt \leq
\frac{1}{ {\left(m+1 \right)}^{2} } \sum_{k=0}^{m-1} {\lambda}_{m, k} \int_{0}^{\, 2\pi \,} \left[ { {f}^{(k+2)} (t) }^{2} - { {f}^{(k+1)} (t) }^{2} \right] dt.
\end{equation*}
\item[\textbf{(c)}]
\begin{equation}\label{wirhigher c}
\begin{split}
0\le&S_{m, 0} \int_{0}^{2 \pi} \left(\dot{f}(t)^{2}-f(t)^{2} \right) d t +\sum_{k=1}^{m-1} S_{m, k} \int_{0}^{2 \pi}\left(f^{(k+1)}(t)+f^{(k-1)}(t)\right)^{2} d t\\
\le& \frac{1}{(m+1)^{2}}\left[S_{m,0} \int_{0}^{2 \pi} \left(\dot{f}(t)^{2}-f(t)^{2} \right) d t+\sum_{k=0}^{m-1} S_{m, k} \int_{0}^{2 \pi}\left(f^{(k+2)}(t)+f^{(k)}(t)\right)^{2} d t\right].
\end{split}
\end{equation}
\end{enumerate}
\end{prop}

\begin{proof}
By the identity $Q_{m+1}(t)=(t-(m+1)^2)Q_m(t)$, we have the recurrence relation
\begin{align*}
{\mathbf{c}}_{m+1,k} ={\mathbf{c}}_{m,k-1}- (m+1)^2{\mathbf{c}}_{m,k}.
\end{align*}
Here, we use the standard convention that $\mathbf{c}_{m,k}=0$ for $k\notin\{0,\cdots,m\}$. By replacing $m$ with $m+1$, the inequality \eqref{wirhigherfrac} then gives
\begin{align*}
0\le& \sum_{k=0}^{m+1} \mathbf{c}_{m+1,k}\int_{0}^{\, 2\pi \,} f^{(k)}(t)^2dt
= \sum_{k=0}^{m+1} \left({\mathbf{c}}_{m,k-1}- (m+1)^2{\mathbf{c}}_{m,k}\right)\int_{0}^{\, 2\pi \,}{ f^{(k)}(t)}^2dt\\
=&\sum_{k=0}^{m}\mathbf{c}_{m,k}\int_{0}^{\, 2\pi \,}{f^{(k+1)}(t)}^2 dt-(m+1)^2\sum_{k=0}^{m}\mathbf{c}_{m,k}\int_{0}^{\, 2\pi \,}{f^{(k)}(t)}^2dt.
\end{align*}
This together with \eqref{wirhigherfrac} proves the second inequality in \eqref{wirhigher a}. The second inequality of \eqref{wirhigher a} becomes an equality when $f$ is equal to the finite Fourier series $\sum_{n=1}^{m+1} {\alpha}_{n} \cos \left(n t \right) + {\beta}_{n} \sin \left(n t \right)$.

The proofs for (b) and (c) are similar. The proof of (b) requires the recurrence relation ${\lambda}_{m+1, k} = \lambda_{m, k-1} - {\left(m+1 \right)}^{2} \lambda_{m, k}$,
which follows from  the identity $P_{m+1}(t)=(t-(m+1)^2) P_{m}(t)$. The proof of (c) requires the recurrence relation  $S_{m+1, k}=S_{m, k-1}-(m+1)^{2} S_{m, k}$, for $k\ge 1$,
which is obtained from  the identity $\mathcal S_{m+1}(t)=\left(t-(m+1)^{2}\right)  \mathcal S_m (t)+S_{m, 0}$.
\end{proof}

\begin{prop}\label{Wir3}
If $h$ is a $2\pi$-periodic function of class ${C}^{m}$, then
\begin{equation} \label{prop3}
\begin{split}
0\le&
\frac{(-1)^m}{2}(m-1)!(m+1)!\int_{0}^{\, 2\pi \,}\left(h(t)^2-\dot{h}(t)^2\right)dt+\frac{(-1)^{m-1}(m!)^2}{\, 2\pi \,}\left(\int_{0}^{\, 2\pi \,}h(t)dt\right)^2\\
&+\sum_{k=1}^{m-1}S_{m,k}\int_{0}^{\, 2\pi \,}\left(h^{(k+1)}(t)+h^{(k-1)}(t)\right)^2 dt
\end{split}
\end{equation}
where $S_{m,k}$ are given by \eqref{secondpolynomial}. The equality holds if and only if $h(t)=\sum_{n=0}^{m}\left(\alpha_n\cos (nt)+\beta_n\sin(n t)\right)$.

\end{prop}
\begin{proof}
Let $c= \frac{1}{\, 2\pi \, } \int_{0}^{\, 2\pi \,} \, {h(t)} \, dt$ be the mean value of $h$. Then the function $f(t):=h(t)-c$ has mean zero.
By a direct computation, we have $S_{m,0}=\frac{(-1)^m}{2}(m-1)!(m+1)!$, ${S_{m,0}-S_{m,1}}=(-1)^{m-1}(m!)^2$,
\begin{align*}
\int_{0}^{2 \pi}\left[\dot{f}(t)^{2}-f(t)^{2}\right] d t= \int_{0}^{2 \pi}\left[\dot{h}(t)^{2}-h(t)^{2}\right] d t+\frac{1}{2 \pi}\left(\int _0^{2\pi}h(t)dt\right)^{2},
\end{align*}
and
\begin{align*}
\int_{0}^{2 \pi}\left(\ddot f(t)+f(t)\right)^{2} d t
=\int_{0}^{2 \pi}\left(\ddot h(t)+h(t)\right)^{2} d t-\frac{1}{2 \pi}\left(\int_{0}^{2 \pi} h(t) d t\right)^{2}.
\end{align*}
Plugging $f(t)=h(t)-c$ into \eqref{wirhigher2} and using the above relations, we can get the result.
\end{proof}

\begin{exmp}
For the convenience of the readers, we write down some special cases of Proposition \ref{Wirs0}, Proposition \ref{Wirs} and Proposition \ref{Wir3} when $m$ is small. These inequalities
will be used in Section \ref{Gapplications}.
\begin{enumerate}
\item[]
\item
We take $m=2$ in Proposition \ref{Wirs0}. We have $P_{2}(t) = t -2^{2}$,  ${\lambda}_{2, 0}=-4$, ${\lambda}_{2, 1}=1$, ${S}_{2, 0}=-3$, and ${S}_{2, 1}=1$.
If $f$ is $ C^2$ and has mean zero, the inequality \eqref{wirhigher2here} implies
\begin{equation} \label{wir01}
0 \leq \int_{0}^{\, 2\pi \,} \, \left({\dot{f}(t)}^{2} - {f(t)}^{2}\right) \, dt \leq \frac{1}{\, 3 \, } \int_{0}^{\, 2\pi \,} \, { \left(\, {\ddot{f}(t)} + f(t) \, \right) }^{2} \, dt.
\end{equation}
The inequality on the left-hand side in \eqref{wir01}, the classical Wirtinger inequality, becomes an equality if and only if
$f(t)={\alpha}_{1} \cos t + {\beta}_{1} \sin t$.
The inequality on the right-hand side in \eqref{wir01}, the reverse Wirtinger inequality, becomes an equality if and only if $f$ has the following expansion:
\begin{equation} \label{wir02equal}
f(t)= \sum_{n=1}^{2} \left({\alpha}_{n} \cos \left(n t \right) + {\beta}_{n} \sin \left(n t \right)\right).
\end{equation}
\item
We take $m=3$ in Proposition \ref{Wirs0}. We have $P_{3}(t) = \left(t -2^{2} \right)\left(t -3^{2} \right)=t^2 -13t + 36$, ${\lambda}_{3, 0}=36$, ${\lambda}_{3, 1}=-13$, ${\lambda}_{3, 2}=1$, ${S}_{3, 0}=24$, ${S}_{3, 1}=-12$, and ${S}_{3, 2}=1$. Assume that
$f$ is $C^3$ and has mean zero. It follows from the inequality \eqref{wirhigher c} that
\begin{equation} \label{wir02h}
\begin{split}
0 \leq&  \frac{1}{\, 3 \, } \int_{0}^{\, 2\pi \,} \, { \left(\, {\ddot{f}(t)} + f(t) \, \right) }^{2} \, dt - \int_{0}^{\, 2\pi \,} \, \left({\dot{f}(t)}^{2} - {f(t)}^{2}\right) \, dt  \\
 \leq&  \frac{1}{\, 36 \, }  \int_{0}^{\, 2\pi \,} \, { \left(\, {\dddot{f}(t)} + \dot{f}(t) \, \right) }^{2} \, dt
  -   \frac{1}{\, 3 \, }  \int_{0}^{\, 2\pi \,} \,  \left({\dot{f}(t)}^{2} - {f(t)}^{2}\right)   \, dt
  \end{split}
\end{equation}
\item[(3)]
Let $h : \mathbb{R} \to \mathbb{R}$ be  a $2\pi$-periodic function of class $C^{2}$. Taking $m=2$ in Proposition \ref{Wir3}, we have
\begin{equation} \label{wir01b}
\begin{split}
0 \leq&    {\left(\int_{0}^{\, 2\pi \,} \, {h(t)} \, dt \right)}^{2} - 2 \pi \int_{0}^{\, 2\pi \,} \, \left[ {h(t)}^{2} - {\dot{h}(t)}^{2} \right] \, dt \\
   \leq&   \frac{2 \pi}{\, 3 \, } \left[ \, \int_{0}^{\, 2\pi \,} \, { \left(\, {\ddot{h}(t)} + h(t) \, \right) }^{2} \, dt - \frac{1}{\, 2\pi \,} {\left(\int_{0}^{\, 2\pi \,} \, {h(t)} \, dt \right)}^{2} \, \right].
\end{split}
\end{equation}
\item[(4)]
Let $h : \mathbb{R} \to \mathbb{R}$ be  a $2\pi$-periodic function of class $C^{3}$. Taking $m=3$ in Proposition \ref{Wir3}, we have
\begin{equation} \label{wir02b}
\begin{split}
0 \leq&    \int_{0}^{\, 2\pi \,} \, { \left(\, {\ddot{h}(t)} + h(t) \, \right) }^{2} \, dt - \frac{1}{\, 2\pi \,} {\left(\int_{0}^{\, 2\pi \,} \, {h(t)} \, dt \right)}^{2}
 -    \frac{3} {\, 2 \pi\,} \left[     {\left(\int_{0}^{\, 2\pi \,} \, {h(t)} \, dt \right)}^{2} - 2 \pi \int_{0}^{\, 2\pi \,} \, \left[ {h(t)}^{2} - {\dot{h}(t)}^{2} \right] \, dt    \right]    \\
   \leq&   \frac{1}{\, 12 \,}  \left[      \int_{0}^{\, 2\pi \,} \, { \left(\, {\dddot{h}(t)} + {\dot{h}}(t) \, \right) }^{2} \, dt   - \frac{6}{\, \pi \,} \left( \,
     {\left(\int_{0}^{\, 2\pi \,} \, {h(t)} \, dt \right)}^{2} - 2 \pi \int_{0}^{\, 2\pi \,} \, \left[ {h(t)}^{2} - {\dot{h}(t)}^{2} \right] \, dt
   \, \right) \right].
\end{split}
\end{equation}
\end{enumerate}
\end{exmp}

\section{Geometric Applications of Wirtinger-type Inequalities} \label{Gapplications}

\begin{thm}\label{first thm}
Let $\mathcal{C}$ be a simple closed $C^2$ curve in ${\mathbb{R}}^{2}$ with length $L$ and a unit speed counter-clockwise
parametrization $\mathbf{X}(s)=\left(x(s), y(s) \right)$ with $s \in \left[0, L\right]$.
Assume that $\mathcal{C}=\partial \Omega$ is the boundary curve of a domain $\Omega \subset {\mathbb{R}}^{2}$ with the area $A$.
Let $\mathbf{T}(s)=\frac{d}{ds} \mathbf{X}(s)$ denote the unit tangent vector field on the curve $\mathcal{C}$. Let $\mathbf{N}=-J\mathbf{T}$ denote the outward-pointing unit normal vector field on the boundary curve $\mathcal{C}$, where $J$ denotes the counter-clockwise $\frac{\, \pi \, }{2}$-rotation. Let $\mathbf{G}$ denote the
the center of the gravity of the curve $\mathcal{C}$:
\begin{equation*} \label{centergravity}
\mathbf{G} =\frac{1}{\, L\, } \int_{\mathcal{C}} \mathbf{X} \, ds.
\end{equation*}
Let $\kappa(s)$ denote the signed curvature of the curve $\mathcal{C}$. Then we have the following geometric inequalities.
\begin{enumerate}
\item[\textbf{(a)}]
\begin{equation} \label{Reverse01}
0 \leq \left(L^2 - 4 \pi A \right) - \frac{\, 2 {\, \pi \,}^{2} \, }{L} \int_{\mathcal{C}} \, {\left\vert \,
\mathbf{X} - \mathbf{G} - {\left(\frac{L}{\, 2\pi \,} \right)} \mathbf{N} \, \right\vert}^{2} \, ds
\leq \frac{\, 2 {\, \pi \,}^{2} \, }{3L} \int_{\mathcal{C}} \, {\left\vert \, \mathbf{X} - \mathbf{G} - {\left(\frac{L}{\, 2\pi \,} \right)}^{2} \kappa \mathbf{N} \, \right\vert}^{2} \, ds
\end{equation}
\item[\textbf{(b)}]
\begin{equation} \label{Reverse02}
0 \leq \frac{\, {L}^{3} \, }{\, 4 {\pi}^{2} \, } - \int_{\mathcal{C}} {\left\vert \mathbf{X} - \mathbf{G}\, \right\vert}^{2} \, ds \leq \frac{ L^4 }{\,64 {\pi}^{4}\,}
\left(\int_{\mathcal{C}} {\, \kappa \,}^{2} \, ds - \frac{4 \pi^2}{L} \right).
\end{equation}
\end{enumerate}
An inequality in \eqref{Reverse01} or \eqref{Reverse02} becomes an equality if and only if the curve $\mathcal C$ is a circle.
\end{thm}

\begin{proof} Without loss of generality, by a translation, we may assume that the center of the gravity of the curve $\mathcal{C}$ is the origin:
\begin{equation*} \label{zerogravity}
\mathbf{G} =\frac{1}{\, L\, } \int_{\mathcal{C}} \mathbf{X} \, ds = {O}_{{\mathbb{R}}^{2}}.
\end{equation*}
We first prove the inequalities \eqref{Reverse01}. The first estimation in \eqref{Reverse01} is the sharpened isoperimetric inequality
\begin{equation*} \label{ChakerianE}
\frac{\, 2 {\, \pi \,}^{2} \, }{L} \int_{\mathcal{C}} \, {\left\vert \,
\mathbf{X} - {\left(\frac{L}{\, 2\pi \,} \right)} \mathbf{N} \, \right\vert}^{2} \, ds \leq L^2 - 4 \pi A,
\end{equation*}
proved by G. D. Chakerian \cite{Chakerian1978}. The first inequality in \eqref{Reverse01} is equivalent to the first inequality in \eqref{Reverse02}. The 
equality holds if and only if the curve $\mathcal C$ is a circle \cite[Section 4.5]{DHT2010}. The second inequality
in \eqref{Reverse01} is equivalent to the reverse isoperimetric inequality
\[
L^2 - 4 \pi A \leq \frac{\, 2 {\, \pi \,}^{2} \, }{L} \int_{\mathcal{C}} \, {\left\vert \,
\mathbf{X} - {\left(\frac{L}{\, 2\pi \,} \right)} \mathbf{N} \, \right\vert}^{2} \, ds + \frac{\, 2 {\, \pi \,}^{2} \, }{3L} \int_{\mathcal{C}} \, {\left\vert \, \mathbf{X}
- {\left(\frac{L}{\, 2\pi \,} \right)}^{2} \kappa \mathbf{N} \, \right\vert}^{2} \, ds,
\]
or equivalently,
\begin{equation} \label{wir01cor}
\frac{\, 2 {\, \pi \,} \, }{L} \int_{\mathcal{C}} \, {\left\vert \,
\mathbf{X} - {\left(\frac{L}{\, 2\pi \,} \right)} \mathbf{N} \, \right\vert}^{2} \, ds + \frac{\, 2 {\, \pi \,} \, }{3L} \int_{\mathcal{C}} \, {\left\vert \, \mathbf{X}
- {\left(\frac{L}{\, 2\pi \,} \right)}^{2} \kappa \mathbf{N} \, \right\vert}^{2} \, ds - \frac{\, L^2 - 4 \pi A \, }{\, \pi \,}
\geq 0.
\end{equation}
To verify the estimation \eqref{wir01cor}, we introduce the new parameter $t=\frac{\, 2 \pi\, }{L} s$ on the curve $\mathcal{C}$ so that
\begin{equation} \label{arc}
{\left(\frac{dx}{dt} \right)}^{2} + {\left(\frac{dy}{dt} \right)}^{2} = { \left(\frac{L}{\, 2 \pi \, } \right)}^{2}.
\end{equation}
This induces the $2\pi$-periodic patch $\mathbf{X}(t)=\left(x(t), y(t) \right)$ of the closed curve $\mathcal{C}$.
We will see that the inequality \eqref{wir01cor} is the consequence of the second inequality in \eqref{wir01}, which can be viewed as a reverse Wirtinger inequality:
\begin{equation} \label{wir01here}
\int_{0}^{\, 2\pi \,} \, \left({\dot{f}(t)}^{2} - {f(t)}^{2} \right)\, dt \leq \frac{1}{\, 3 \, } \int_{0}^{\, 2\pi \,} \, { \left(\, {\ddot{f}(t)} + f(t) \, \right) }^{2} \, dt.
\end{equation}
We observe the following equalities.
\begin{enumerate}
\item[\textbf{(i)}]
\[
\begin{bmatrix}
\, \frac{d^{2} x}{ds^2} \, \\ \\ \frac{d^{2} y}{ds^2}
\end{bmatrix} = {\triangle}_{\mathcal{C}} \mathbf{X} = - \kappa \mathbf{N} = \kappa \left(J \mathbf{T} \right) = \kappa \begin{bmatrix}
\, - \frac{dy}{ds} \, \\ \\ \frac{dx}{ds}
\end{bmatrix}.
\]
\item[\textbf{(ii)}]
\begin{eqnarray*}
{I}_{1} &:=& \frac{\, 2 {\, \pi \,} \, }{L} \int_{\mathcal{C}} \, {\left\vert \,
\mathbf{X} - {\left(\frac{L}{\, 2\pi \,} \right)} \mathbf{N} \, \right\vert}^{2} \, ds \\
&=& \frac{\, 2 {\, \pi \,} \, }{L} \int_{0}^{L} \, {\left(x - \frac{L}{\, 2\pi\, } \frac{dy}{ds} \right)}^{2} + {\left(y + \frac{L}{\, 2\pi\, } \frac{dx}{ds} \right)}^{2} \, ds \\
&=& \int_{0}^{\, 2\pi \,} \, {\left(x- \frac{dy}{dt} \right)}^{2} + {\left(y+ \frac{dx}{dt} \right)}^{2} \, dt.
\end{eqnarray*}
\item[\textbf{(iii)}]
\begin{eqnarray*}
{I}_{2} &:=& \frac{\, 2 {\, \pi \,} \, }{L} \int_{\mathcal{C}} \, {\left\vert \, \mathbf{X}
- {\left(\frac{L}{\, 2\pi \,} \right)}^{2} \kappa \mathbf{N} \, \right\vert}^{2} \, ds \\
&=& \frac{\, 2 {\, \pi \,} \, }{L} \int_{0}^{L} \, {\left(x - {\left(\frac{L}{\, 2\pi \,} \right)}^{2} \kappa
\frac{dy}{ds} \right)}^{2} + {\left(y + {\left(\frac{L}{\, 2\pi \,} \right)}^{2} \kappa \frac{dx}{ds} \right)}^{2} \, ds \\
&=& \frac{\, 2 {\, \pi \,} \, }{L} \int_{0}^{L} \, {\left(x + {\left(\frac{L}{\, 2\pi \,} \right)}^{2} \frac{d^{2} x}{ds^2} \right)}^{2} + {\left(y + {\left(\frac{L}{\, 2\pi \,} \right)}^{2} \frac{d^{2} y}{ds^2} \right)}^{2} \, ds \\
&=&  \int_{0}^{\, 2\pi \,} \, {\left(x+ \frac{d^{2} x}{dt^2} \right)}^{2} + {\left(y+ \frac{d^{2} y}{dt^2} \right)}^{2} \, dt.
\end{eqnarray*}
\item[\textbf{(iv)}]
We use the property \eqref{arc} to deduce
\[
\frac{L^2}{\, 2\pi \, } = \int_{0}^{\, 2\pi \,} \, {\left(\frac{L}{\, 2\pi \,} \right)}^{2} \, dt = \int_{0}^{\, 2\pi \,} \, {\left(\frac{dx}{dt} \right)}^{2} +
{\left(\frac{dy}{dt} \right)}^{2} \, dt.
\]
\item[\textbf{(v)}]
Since $\mathcal{C}$ is positively oriented by $\mathbf{X}(s)$, the divergence theorem gives
\begin{equation} \label{areaintegral}
2A = \iint_{\Omega} \, \mathrm{div}\, \mathbf{X} \, ds = \int_{\mathcal{C}} \langle \mathbf{X}, \mathbf{N} \rangle \, ds
= \int_{\mathcal{C}} \langle \mathbf{X}, - J \mathbf{T} \rangle \, ds
= \int_{0}^{\, 2\pi \,} \, \left(x \frac{dy}{dt} - y \frac{dx}{dt} \right) \, dt.
\end{equation}
\item[\textbf{(vi)}]
We compute the isoperimetric deficit $L^2 - 4 \pi A$:
\begin{eqnarray*}
\frac{L^2 - 4 \pi A}{\, \pi \,}
&=& 2\left(\frac{L^2}{\, 2\pi \, } - 2A \right) \\
&=& 2 \left(\int_{0}^{\, 2\pi \,} \, {\left(\frac{dx}{dt} \right)}^{2} + {\left(\frac{dy}{dt} \right)}^{2} \, dt -
\int_{0}^{\, 2\pi \,} \, \left(x \frac{dy}{dt} - y \frac{dx}{dt} \right) \, dt \right) \\
&=& \int_{0}^{\, 2\pi \,} \left[ {\left(x- \frac{dy}{dt} \right)}^{2} + {\left(y+ \frac{dx}{dt} \right)}^{2} + \left({\left(\frac{dx}{dt} \right)}^{2} - x^{2} \right)
+ \left({\left(\frac{dy}{dt} \right)}^{2} - y^{2} \right) \, \right] \, dt.
\end{eqnarray*}
\end{enumerate}
Combining these equalities, we have
\begin{equation}\label{combine}
\begin{split}
& \frac{\, 2 {\, \pi \,} \, }{L} \int_{\mathcal{C}} \, {\left\vert \,
\mathbf{X} - {\left(\frac{L}{\, 2\pi \,} \right)} \mathbf{N} \, \right\vert}^{2} \, ds + \frac{\, 2 {\, \pi \,} \, }{3L} \int_{\mathcal{C}} \, {\left\vert \, \mathbf{X}
- {\left(\frac{L}{\, 2\pi \,} \right)}^{2} \kappa \mathbf{N} \, \right\vert}^{2} \, ds - \frac{\, L^2 - 4 \pi A \, }{\, \pi \,} \\
=& {I}_{1}+ \frac{1}{\, 3\, } {I}_{2} - \frac{\, L^2 - 4 \pi A \, }{\, \pi \,} \\
=& \int_{0}^{\, 2\pi \,} \, \left[ \, {\left(x- \frac{dy}{dt} \right)}^{2} + {\left(y+ \frac{dx}{dt} \right)}^{2} \, \right] \, dt +
\frac{1}{\, 3\, } \int_{0}^{\, 2\pi \,} \, \left[ \, {\left(x+ \frac{d^{2} x}{dt^2} \right)}^{2} + {\left(y+ \frac{d^{2} y}{dt^2} \right)}^{2} \, \right] \, dt \\
& \, - \int_{0}^{\, 2\pi \,} \left[ {\left(x- \frac{dy}{dt} \right)}^{2} + {\left(y+ \frac{dx}{dt} \right)}^{2} + \left({\left(\frac{dx}{dt} \right)}^{2} - x^{2} \right)
+ \left({\left(\frac{dy}{dt} \right)}^{2} - y^{2} \right) \, \right] \, dt \\
=& \int_{0}^{\, 2\pi \,} \left[ \, \frac{1}{\, 3\, } {\left(x+ \frac{d^{2} x}{dt^2} \right)}^{2} - \left({\left(\frac{dx}{dt} \right)}^{2} - x^{2} \right) \, \right] \, dt
+ \int_{0}^{\, 2\pi \,} \left[ \, \frac{1}{\, 3\, } {\left(y+ \frac{d^{2} y}{dt^2} \right)}^{2} - \left({\left(\frac{dy}{dt} \right)}^{2} - y^{2} \right) \, \right] \, dt \\
\geq& 0.
\end{split}
\end{equation}
In the last line, we exploited the reverse Wirtinger inequality \eqref{wir01here} twice.

We keep assuming $\mathbf{G} =\frac{1}{\, L\, } \int_{\mathcal{C}} \mathbf{X} \, ds = {O}_{{\mathbb{R}}^{2}}$ and verify the inequalities in
\eqref{Reverse02}:
\begin{equation} \label{wir02here}
0 \leq \frac{\, {L}^{3} \, }{\, 4 {\, \pi \,}^{2} \, } - \int_{\mathcal{C}} {\left\vert \mathbf{X} \, \right\vert}^{2} \, ds \leq \frac{ L^4 }{64 {\, \pi \,}^{2} }
\left(\int_{\mathcal{C}} {\, \kappa \,}^{2} \, ds - \frac{4 \pi^2}{L} \right).
\end{equation}
The first inequality is due to Sachs \cite{Sachs1960}, and is used to prove the isoperimetric inequalities. For instance, see \cite{Chakerian1978, DHT2010, LSY1984, Osserman1978}. It now remains to establish the second inequality, which can be viewed as a reverse Sachs Inequality. We will show that
the second inequality in \eqref{wir02here} is equivalent to the second inequality in \eqref{Reverse01},
which can be written as
\begin{equation} \label{wir01cor2}
\int_{\mathcal{C}} \, {\left\vert \,
\mathbf{X} - {\left(\frac{L}{\, 2\pi \,} \right)} \mathbf{N} \, \right\vert}^{2} \, ds + \frac{\, 1 \, }{3} \int_{\mathcal{C}} \, {\left\vert \, \mathbf{X}
- {\left(\frac{L}{\, 2\pi \,} \right)}^{2} \kappa \mathbf{N} \, \right\vert}^{2} \, ds - \frac{\, L \, }{2 {\, \pi \,}^{2} } \left(L^2 - 4 \pi A \right)
\geq 0.
\end{equation}
The Minkowski formula (for instance, see \cite{Klingenberg2013})
\begin{equation*} \label{lengthint2}
L = \int_{\mathcal{C}} \kappa {\langle \mathbf{X}, \mathbf{N} \rangle} \, ds
\end{equation*}
implies
\begin{equation} \label{square01}
\int_{\mathcal{C}} \, {\left\vert \, \mathbf{X}
- {\left(\frac{L}{\, 2\pi \,} \right)}^{2} \kappa \mathbf{N} \, \right\vert}^{2} \, ds
= \int_{\mathcal{C}} \, {\left\vert \, \mathbf{X} \right\vert}^{2} \, ds - \frac{2 L^3}{{\left(2 \pi \right)}^{2}} +
{\left(\frac{L}{\, 2\pi \,} \right)}^{4} \int_{\mathcal{C}} {\, \kappa \,}^{2} \, ds.
\end{equation}
We also use the geometric identity, deduced in \eqref{areaintegral}, 
\[
2A= \int_{\mathcal{C}} {\langle \mathbf{X}, \mathbf{N} \rangle} \, ds
\]
to have
\begin{equation} \label{square02}
\int_{\mathcal{C}} \, {\left\vert \, \mathbf{X} - {\left(\frac{L}{\, 2\pi \,} \right)} \mathbf{N} \, \right\vert}^{2} \, ds
= \int_{\mathcal{C}} \, {\left\vert \, \mathbf{X} \right\vert}^{2} \, ds - \frac{2AL}{\, \pi \,} + \frac{L^3}{{\left(2 \pi \right)}^{2}}
\end{equation}
We plug identities \eqref{square01} and \eqref{square02} into the left hand side in \eqref{wir01cor2} to deduce
\begin{eqnarray*}
&& \int_{\mathcal{C}} \, {\left\vert \, \mathbf{X} - {\left(\frac{L}{\, 2\pi \,} \right)} \mathbf{N} \, \right\vert}^{2} \, ds + \frac{\, 1 \, }{3} \int_{\mathcal{C}} \, {\left\vert \, \mathbf{X}
- {\left(\frac{L}{\, 2\pi \,} \right)}^{2} \kappa \mathbf{N} \, \right\vert}^{2} \, ds - \frac{\, L \, }{2 {\, \pi \,}^{2} } \left(L^2 - 4 \pi A \right) \\
&=& \frac{1}{3} {\left(\frac{L}{\, 2\pi \,} \right)}^{4} \left[ \, \int_{\mathcal{C}} {\, \kappa \,}^{2} \, ds - \frac{4 \pi^2}{L} \, \right] - \frac{4}{3} \left[ \,
\frac{\, {L}^{3} \, }{\, 4 {\, \pi \,}^{2} \, } - \int_{\mathcal{C}} {\left\vert \mathbf{X} \, \right\vert}^{2} \, \right],
\end{eqnarray*}
which shows the equivalence of the second inequality in \eqref{wir02here} and the second inequality in \eqref{wir01cor2}.

Finally, we check the equality cases. Suppose that the second equality in \eqref{Reverse01} holds.  By \eqref{combine} and \eqref{wir02equal}, the
$\mathbb{C}$-valued $2\pi$-period function $z(t)=x'(t)+iy'(t)$ has to satisfy
\[
z(t)=\sum_{0<|n|\le 2} a_n e^{int} = a_{-2} e^{-2it} +  a_{-1} e^{-it}  +   a_{1} e^{it} + a_{2} e^{2it},
\]
for some constants $a_{-2}, a_{-1}, a_{1}, a_{2}  \in \mathbb{C}$. Also, by the assumption \eqref{arc}, $|z(t)|^2=\mu$ is a constant.
Then
\[
\mu =  z(t) \overline{z(t)}
=\sum_{n=-4}^{4} \, \left(\sum_{p+q=n} a_p\overline {a_{-q}} \right) \, e^{int}.
\]
The uniqueness of Fourier series shows that, for all $n \in \left\{-4, -3, -2, -1, 1, 2, 3, 4\right\}$, the coefficients vanish:
\[
\sum_{p+q=n, \, 1  \leq |p|, |q| \leq 2} a_p\overline {a_{-q}} =0.
\]
In particular, we have
\[
\quad a_{2} \overline{a_{1}} + a_{-1}  \overline{a_{-2}} =0, \quad a_{1} \overline{a_{-1}}=0, \quad
a_{2} \overline{a_{-1}} + a_{1}  \overline{a_{-2}} =0, \quad  a_{2} \overline{a_{-2}}=0.
\]
These equalities imply that at most one of the coefficients $a_{-2}, a_{-1}, a_{1}, a_{2}$ is non-zero. By applying a rigid motion in ${\mathbb{R}}^{2}=\mathbb{C}$, we can assume $x'(t)+iy'(t)= R e^{it}$ or $x'(t)+iy'(t)=R e^{i \left(2t\right)}$ for some constant $R>0$. We recall that the simple closed curve $\mathcal C$ is parameterized by the $2\pi$-period patch $x(t)+iy(t)$. We conclude that the curve $\mathcal C$ is a circle.
\end{proof}

\begin{thm}\label{second thm}
Let $\mathcal{C}$ be a simple closed $C^{m}$ convex curve with length $L$.
Assume that $\mathcal{C}=\partial\Omega$ is the boundary of the domain $\Omega \subset {\mathbb{R}}^{2}$ with area $A$.
Let $\kappa>0$ be the curvature of $\mathcal C$ and $\rho=\frac{1}{\, \kappa \,}$ be the radius of curvature. Let $S_{m,k}$ be defined by \eqref{secondpolynomial}.
\begin{enumerate}
\item
If $m$ is an odd positive integer,
\begin{align*}
L^2-4\pi A\ge \frac{(m-1)\pi}{m} \left(\int_{\mathcal C}\frac{1}{\, \kappa \,}ds- 2A\right)-\frac{\, 2\pi \,}{(m!)^2}\sum_{l=1}^{m-2}S_{m,l+1}
\int_{\mathcal C}\frac {1} {\rho(s)}\left[\left(\rho(s)\frac{d}{ds}\right)^l \rho(s)\right]^2 ds.
\end{align*}
\item
If $m$ is an even positive integer,
\begin{align*}
L^2-4\pi A
\le \frac{(m-1)\pi}{m} \left(\int_{\mathcal C}\frac{1}{ \kappa }ds- 2A\right)+\frac{\, 2\pi \,}{(m!)^2}\sum_{l=1}^{m-2}S_{m,l+1}
\int_{\mathcal C}\frac {1} {\rho(s)}\left[\left(\rho(s)\frac{d}{ds}\right)^l \rho(s)\right]^2 ds.
\end{align*}
\end{enumerate}
\end{thm}

\begin{proof}
We use the notation in the proof of Theorem \ref{first thm}. Let $\theta$ denote the normal angle function on the convex curve $\mathcal{C}$.
We define the support function $h=h\left(\theta\right)=\langle \mathbf{X}, \mathbf{N} \rangle$. We have the equalities \cite[p. 34]{ChouZhu2001}
\begin{equation}\label{rho}
\kappa ds = d\theta, \quad \frac{d}{d\theta}=\frac{1}{\, \kappa\, } \frac{d}{\, ds\, }, \quad \frac{1}{\, \kappa \,} = h+ \frac{d^2}{d {\theta}^{2}} {h}.
\end{equation}
By the Minkowski formula, we also have
\[
L = \int_{\mathcal{C}} \, {\langle \mathbf{X}, \mathbf{N} \rangle} \kappa \, ds = \int_{0}^{\, 2\pi \,} h \, d\theta
\]
and
\[
2A= \int_{\mathcal{C}} \, {\langle \mathbf{X}, \mathbf{N} \rangle} \, ds = \int_{0}^{\, 2\pi \,} \, \frac{h}{\, \kappa \,} \, d\theta
= \int_{0}^{\, 2\pi \,} \, h \left(h+ \frac{d^2}{d {\theta}^{2}} {h} \right) \, d\theta
= \int_{0}^{\, 2\pi \,} \, \left(h^2 - {\left(\frac{dh}{d \theta} \right)}^{2} \right)\, d\theta.
\]
We compute ${S_{m,1}}/\left(\frac{(-1)^{m-1}(m!)^2}{\, 2\pi \,}\right)=\frac{(1-m)\pi}{m}$. Putting all these into \eqref{prop3}, we arrive at the results.
\end{proof}
\begin{rmk}[Equality cases of Theorem \ref{second thm}]
We regard $\mathcal C$ as a curve on the complex plane $\mathbb{C}={\mathbb{R}}^{2}$. Suppose that the inequality becomes an equality, then by Proposition \ref{Wir3}, 
\[
 h(\theta)=\sum_{n=-m}^{m}a_n e^{in \theta}
\]
with $a_{-n}=\overline {a_n}$ (as $h$ is real-valued). Here, the normal and the tangent vectors are represented by $e^{i\theta}$ and $ie^{i\theta}$ respectively.
By identifying $\mathbb R^2$ with $\mathbb C$, we recover $\mathcal C$ from $h$ by the parametrization  \cite[p. 34]{ChouZhu2001}
\[
\mathbf X(\theta)=h(\theta)(\cos \theta,\sin \theta)+h'(\theta) (-\sin \theta,\cos \theta)=h e^{i\theta}+i h'(\theta)e^{i\theta}= \sum_{n=-m}^{m}a_n(1-n)e^{i(n+1)\theta}.
\]
 It is not known to us whether there is a simple algebraic condition on the coefficients $a_n$ to guarantee that $\mathcal C$ is a convex curve.  By \eqref{rho}, the convexity of $\mathcal C$ is equivalent to 
 the analytic condition
 \[
 0<h''(\theta)+h(\theta)=a_{0}-2 \sum_{n=2}^{m}\left(n^{2}-1\right) \operatorname{Re}\left(a_{n} e^{i n \theta}\right).
 \]
 For an explicit example of non-round convex curves, we refer the interested readers to \cite[pp. 387]{Groemer1990}.
\end{rmk}

\begin{exmp} \label{coro second main}
Let us examine the inequalities in Theorem \ref{second thm} for $m=1,2,3$.
\begin{enumerate}
\item
When $m=1$, it reduces to the classical isoperimetric inequality $L^2-4\pi A\ge 0$.
\item
When $m=2$, this recovers  the Lin-Tsai inequality \cite[Lemma 1.7]{LinTsai2012} for the convex curve $\mathcal{C}$
\begin{equation*}\label{convex01}
\int_{\mathcal{C}} \frac{1}{\, \kappa \,} \, ds - 2A \geq \frac{2}{\, \pi \,} \left({L}^{2} - 4 \pi A \right),
\end{equation*}
or equivalently, 
\begin{equation} \label{lin tsai}
 L^{2} - 4\pi A \leq \frac{\, 2\pi\, }{3} \left(\int_{\mathcal{C}} \frac{1}{\, \kappa \, } ds - \frac{L^2}{\, 2\pi \,} \right).
\end{equation}
\item
Taking $m=3$ in Theorem \ref{second thm} gives a measure of the stability of Lin-Tsai's inequality:
\begin{equation}\label{convex02}
0 \leq \left(\int_{\mathcal{C}} \frac{1}{\, \kappa\, } ds - \frac{L^2}{\, 2\pi \,} \right) - \frac{3}{\, 2\pi\, } \left(L^{2} - 4\pi A \right)
\leq \frac{1}{\, 12\, } \left[ \, \int_{\mathcal{C}} \frac{1}{\, {\, \kappa \,}^{5} \, } {\left(\frac{d\kappa}{ds} \right)}^{2} ds - \frac{6}{\, \pi\, } \left(L^{2} - 4\pi A \right) \, \right].
\end{equation}
\end{enumerate}
\end{exmp}

\begin{cor}[\textbf{Reverse isoperimetric inequalities}]\label{cor reverse}
Let $\mathcal{C}$ be a convex curve with the length $L$ and the curvature function $\kappa>0$.
Assume that $\mathcal{C}=\partial\Omega$ is the boundary of the domain $\Omega \subset {\mathbb{R}}^{2}$ with the area $A$.
Then, we have the following geometric inequalities
\begin{enumerate}
\item[\textbf{(a)}]
\begin{equation} \label{reverse01}
L^{2} - 4\pi A \leq \frac{\, \pi\, }{6} \int_{\mathcal{C}} \frac{1}{\, {\, \kappa \,}^{5} \, } {\left(\frac{d\kappa}{ds} \right)}^{2} ds.
\end{equation}
\item[\textbf{(b)}]
\begin{equation} \label{reverse02}
L^{2} - 4\pi A \leq \frac{L^2}{\, 24\pi \,} \int_{\mathcal{C}} \frac{1}{\, {\, \kappa \,}^{3} \, } {\left(\frac{d\kappa}{ds} \right)}^{2} ds.
\end{equation}
\item[\textbf{(c)}]
\begin{equation} \label{reverse03}
L^{2} - 4\pi A \leq \frac{\, AL \,}{\, 4 \pi \,} \int_{\mathcal{C}} \frac{1}{\, {\, \kappa \,}^{2} \, } {\left(\frac{d\kappa}{ds} \right)}^{2} ds.
\end{equation}
\end{enumerate}
\end{cor}

\begin{proof}
The two inequalities in \eqref{convex02} implies the reverse isoperimetric inequality \eqref{reverse01}.
J. Bernstein and T. Mettler \cite[(1.3) and Theorem 1.1]{BerMet2012} showed the following particular case of Benguria-Loss conjecture:
\begin{equation*} \label{BM02}
\frac{1}{\, 4 \, } \int_{\mathcal{C}} \frac{1}{\, {\, \kappa \,}^{3} \, } {\left(\frac{d\kappa}{ds} \right)}^{2} ds \geq - 2 \pi + \frac{ {\left(2\pi \right)}^{2} }{ {L}^{2} }
\int_{\mathcal{C}} \frac{1}{\, \kappa \,} \, ds.
\end{equation*}
Combining this and the second inequality in \eqref{lin tsai}
\begin{equation*} \label{LT}
\int_{\mathcal{C}} \frac{1}{\, \kappa \, } ds \geq \frac{L^2}{\, 2\pi \,} + \frac{3}{\, 2\pi\, } \left(L^{2} - 4\pi A \right)
\end{equation*}
yields the reverse isoperimetric inequality \eqref{reverse02}. Indeed, we have
\[
\frac{1}{\, 4 \, } \int_{\mathcal{C}} \frac{1}{\, {\, \kappa \,}^{3} \, } {\left(\frac{d\kappa}{ds} \right)}^{2} ds \geq
- 2 \pi + \frac{ {\left(2\pi \right)}^{2} }{ {L}^{2} } \left[ \frac{L^2}{\, 2\pi \,} + \frac{3}{\, 2\pi\, } \left(L^{2} - 4\pi A \right) \right]
= \frac{\, 6\pi\, }{{L}^{2}} \left(L^{2} - 4\pi A \right).
\]
Combining the Bernstein-Mettler inequality \cite[(1.4) and Theorem 1.1]{BerMet2012}
\begin{equation*} \label{BM02}
\frac{1}{\, 4 \, }\int_{\mathcal{C}} \frac{1}{\, {\, \kappa \,}^{2} \, } {\left(\frac{d\kappa}{ds} \right)}^{2} ds \geq - \frac{ {\left(2\pi \right)}^{2} }{ {L} } +
\int_{\mathcal{C}} {\, \kappa \,}^{2} \, ds.
\end{equation*}
and the Gage inequality \cite{Gage1983}
\begin{equation*} \label{Gage}
\int_{\mathcal{C}} {\, \kappa \,}^{2} \, ds \geq \frac{\pi L}{A}
\end{equation*}
yields the reverse isoperimetric inequality \eqref{reverse03}. Indeed, we have
\[
\frac{1}{\, 4 \, }\int_{\mathcal{C}} \frac{1}{\, {\, \kappa \,}^{2} \, } {\left(\frac{d\kappa}{ds} \right)}^{2} ds \geq - \frac{ {\left(2\pi \right)}^{2} }{ {L}} +
\frac{\pi L}{A} = \frac{\, \pi \,}{AL} \left(L^{2} - 4\pi A \right).
\]
\end{proof}

 \bigskip
 

\bigskip

\begin{thebibliography}{99}

\bibitem{BeckenbachBellman1961} E.   F.   Beckenbach   and   R.   Bellman,   \emph{Inequalities},   Ergebnisse   der   Mathematik   und   ihrer    Grenzgebiete,   N.   F.,   Bd.   30   Springer-Verlag,   Berlin-G\"{o}ttingen-Heidelberg   1961.

\bibitem{BerMet2012} J.   Bernstein   and   T.   Mettler,   \emph{One-Dimensional   Projective   Structures,   Convex   Curves   and   the   Ovals   of   Benguria   and   Loss},   Commun.   Math.   Phys.    336   (2015),   no.   2,   933--952.

\bibitem{Blaschke1949} W.   Blaschke,   \emph{Kreis   und   Kugel},   Chelsea   Publishing   Co.,   New   York,   1949. 

\bibitem{Chakerian1978} G.   D.   Chakerian,   \emph{The   isoperimetric   theorem   for   curves   on   minimal   surfaces},   Proc.   Amer.   Math.   Soc.   69   (1978),   no.   7,   312--313.

\bibitem{Chavel1978} I.   Chavel,    \emph{On   A.   Hurwitz'   method   in   isoperimetric   inequalities},   Proc.   Amer.   Math.   Soc.   71   (1978),   no.   2,   275--279.

\bibitem{Chavel2001} I.   Chavel,    \emph{Isoperimetric   inequalities:   differential   geometric   and   analytic   perspectives},   Vol.   145.   Cambridge   University   Press,   2001.

\bibitem{ChouZhu2001} K.-S.   Chou   and   X.-P.   Zhu,  \emph{The   curve   shortening   problem},   CRC   Press,   2001.

\bibitem{DAngelo2019} J.   P.   D'Angelo, \emph{Hermitian   Analysis:   From   Fourier   series   to   Cauchy-Riemann geometry},   Birkh\"{a}user   Springer,   Cham,   2019.

\bibitem{DHT2010} U.   Dierkes,   S.   Hildebrandt,   and   A.   J.   Tromba,  \emph{Regularity of minimal surfaces},   Revised   and   enlarged   second   edition,   With   assistance
   and   contributions   by   A.   K\"{u}ster,   Grundlehren   der   Mathematischen   Wissenschaften,   340.   Springer,   Heidelberg,   2010.

\bibitem{Gage1983} M.   E.   Gage,   \emph{An   isoperimetric   inequality   with   applications   to   curve   shortening},   Duke   Math.   J.   50   (1983),   no.   4,   1225--1229.

\bibitem{GageHamilton1986} M.   Gage   and   R.   S.   Hamilton,   \emph{The   heat   equation   shrinking   convex   plane   curves},   J.   Differential   Geom.   23   (1986),   no.   1,   69--96.

\bibitem{GelineauZeng2010} Y.   Gelineau   and   J.   Zeng,   \emph{Combinatorial   interpretations   of   the   Jacobi-Stirling   numbers},   Electron.   J.   Combin.   17   (2010),   no.  1,   Research   Paper   70,   17pp. 

\bibitem{Groemer1990} H.   Groemer,   \emph{Stability   properties   of   geometric   inequalities},   Amer.   Math.   Monthly   97   (1990),   no.   5,   382--394. 

\bibitem{Groemer1996} H.   Groemer,   \emph{Geometric   applications   of   Fourier   series   and   spherical   harmonics},   Cambridge   University   Press,   1996.

\bibitem{HLP1988} G.   H.   Hardy,   J.   E.   Littlewood,   and   G.   P\'{o}lya,   \emph{Inequalities},   Reprint   of   the   1952   edition.   Cambridge   Mathematical   Library.
   Cambridge   University   Press,   Cambridge,   1988. 

\bibitem{Hurwitz1902} A.   Hurwitz,   \emph{Sur   quelques   applications   g\'{e}om\'{e}triques   des   s\'{e}ries   de   Fourier},   Ann.   Sci.   \'{E}cole   Norm.   Sup.   (3)   
19   (1902),   357--408.

\bibitem{Klingenberg2013} W.   Klingenberg,   \emph{A course   in   differential   geometry},   volume   51.   Springer   Science   \&   Business   Media,   2013.

\bibitem{LinTsai2012} Y.-C.   Lin   and   D.-H.   Tsai, \newblock    \emph{Application   of   Andrews   and   Green-Osher   inequalities   to   nonlocal   flow   of   convex   plane   curves}, \newblock   J.   Evol.   Equ.   12   (2012),   no.   4,   833--854.

\bibitem{LinTsai2014} Y.-C.   Lin   and   D. H.   Tsai, \newblock    \emph{Asymptotic   behavior   of   the   isoperimetric   deficit   for   expanding   convex   plane   curves},  \newblock   J.   Evol.   Equ.   14   (2014),   no.   4-5,   779--794. 

\bibitem{LSY1984} P.   Li,   R.   Schoen,   and   S.   -T.   Yau,   \emph{On   the   isoperimetric   inequality   for   minimal   surfaces},   Ann.   Scuola   Norm.   Sup.   Pisa   Cl.   Sci.   11   (1984),   no.   2,   237--244.

\bibitem{MPF1991} D.   S.   Mitrinovi\'{c},   J.   E.   Pe\v{c}ari\'{c},    A.   M.   Fink,   \emph{Inequalities   involving   functions   and   their   integrals   and   derivatives},   Mathematics   and   its   Applications   (East   European   Series),   53.   Kluwer   Academic   Publishers   Group,   Dordrecht,   1991. 

\bibitem{Osserman1978} R.   Osserman,   \emph{The   isoperimetric   inequality},   Bull.   Amer.   Math.   Soc.   84   (1978),   no.   6,   1182--1238.

\bibitem{Reilly1977} R.   C.   Reilly,   \emph{On   the   first   eigenvalue   of   the   Laplacian   for   compact   submanifolds   of   Euclidean   space},   Comment.   Math.   Helv.   52
   (1977),   no.   4,   525--533.

\bibitem{Sachs1960} H.   Sachs,   \emph{Ungleichungen   f\"{u}r   Umfang,   Fl\"{a}cheninhalt   und   Tr\"{a}gheitsmoment   konvexer   Kurven},   Acta   Math.   Acad.   Sci.   Hungar.   11   (1960),   103--115.

\bibitem{SteinShakarchi2011} E.   M.   Stein   and   R.   Shakarchi,   \emph{Fourier   analysis:   An   introduction},    Princeton   Lectures   in   Analysis,   1.   Princeton   University   Press,    Princeton,   NJ,   2003.

\bibitem{Tsai2005} D.-H.   Tsai,   \emph{Asymptotic   closeness   to   limiting   shapes   fo    expanding   embedded   plane   curves},   Invent.   Math.   162   (2005),   no.   3,   473--492.

\end{thebibliography}
\end{document}